\documentclass{article}
\usepackage[english]{babel}
\usepackage{graphicx,multicol}
\usepackage{epic,eepic,epsfig}
\usepackage{amssymb}
\usepackage{amsmath,amsthm,authblk}
\usepackage{color}
\usepackage{tikz}
\usetikzlibrary{shapes,snakes}
\usetikzlibrary{arrows.meta, positioning, fit, shapes.geometric, calc}

\usepackage{epsfig,url}
\usepackage{multirow}

\usepackage{amsfonts}
\usepackage{setspace}
\usepackage{geometry}
\usetikzlibrary{shapes,snakes}

\usepackage{eso-pic} 

  \newcommand{\noCol}[1]{}

\newcommand{\be}{\begin{enumerate}}
\newcommand{\ee}{\end{enumerate}}
\newcommand{\bd}{\begin{description}}
\newcommand{\ed}{\end{description}}

\newcommand{\beq}{\begin{equation}}
\newcommand{\eeq}{\end{equation}}

\newcommand{\2}{\vspace{2mm}}

\renewenvironment{proof}[1][]{\par \noindent {\bf Proof#1}.\ }{\hfill$\Box$
\par \vspace{11pt}}

\newtheorem{theorem}{Theorem}[section]
\newtheorem{lemma}[theorem]{Lemma}

\newtheorem{corollary}[theorem]{Corollary}
\newtheorem{remark}[theorem]{Remark}

\theoremstyle{definition}
\newtheorem{definition}[theorem]{Definition}

\newtheorem{problem}[theorem]{Problem}

\begin{document}
\bibliographystyle{plain}

\title{Backward Arcs in Hamilton Oriented  Cycles and Paths in Directed Graphs with  Independence Number Two}
\author[1]{S.~Gerke}
\author[2]{Q.~Guo}
\author[2,3]{G.~Gutin}
\author[3]{Y.~Hao}
\author[1]{W.~Veeranonchai}
\author[4,5]{A.~Yeo}

\affil[1]{Department of Mathematics, Royal Holloway University of London, UK}
\affil[2]{Department of Computer Science, Royal Holloway University of London, UK}
\affil[3]{Center for Combinatorics and LPMC, Nankai University, China}
\affil[4]{Department of Mathematics and Computer Science, University of Southern Denmark, Denmark}
\affil[5]{Department of Mathematics, University of Johannesburg, South Africa}

\date{\today}
\maketitle

\begin{abstract}
In a digraph $D=(V,A)$, an oriented path is a sequence $P=x_1x_2\dots x_p$ of distinct vertices such that either $x_ix_{i+1}\in A$ or $x_{i+1}x_{i}\in A$ or both for every $i\in [p-1]$. If $x_ix_{i+1}\in A$ in $P$, then $x_ix_{i+1}$ is a forward arc of $P$;
otherwise, $x_{i+1}x_{i}$ is a backward arc. The independence number $\alpha(D)$ is the maximum integer $p$ such that $D$ has a set of $p$ vertices where there is no arc between any pair of vertices. 
A digraph is $k$-connected if its underlying undirected graph is $k$-connected. 
Freschi and Lo (JCT-B 2024) proved that every $n$-vertex oriented graph with minimum degree $\delta\ge n/2$ has a Hamilton oriented cycle with at most $n-\delta$ backward arcs. We prove that 
every 2-connected digraph $D$ with $\alpha(D)\le 2$ has a Hamilton oriented cycle with at most five backward arcs, and  every 1-connected digraph $D$ with $\alpha(D)\le 2$ has a Hamilton oriented path with at most two backward arcs.
\end{abstract}

\section{Introduction}\label{sec:intro}


In a digraph $D=(V,A)$, an {\em oriented path} (or, just {\em orpath}) is a sequence $P=x_1x_2\dots x_p$ of distinct vertices such that either $x_ix_{i+1}\in A$ or $x_{i+1}x_{i}\in A$ or both for every $i\in [p-1]$. If $x_ix_{i+1}\in A$ in $P$, then $x_ix_{i+1}$ is a {\em forward arc} of $P$ and if  $x_{i+1}x_{i}\in A$, then  $x_{i+1}x_{i}$ is a {\em backward arc} of $P.$  An orpath $P$ is a {\em directed path} (or, just a {\em dipath}) if all arcs of $P$ are forward. An orpath is {\em Hamilton} if it contains all vertices of $D$. Similarly, we can define orcycles and dicycles. For an orpath or orcycle $R$, let $\sigma^+(R)$ and $\sigma^-(R)$ be the {\em number of forward and backward arcs}, respectively. 


Erd{\H o}s \cite{Erdos63} introduced the notion of graph discrepancy for 2-edge-colored undirected graphs in 1963.
His paper has launched an extensive research on the topic, see e.g. references in \cite{FL24,GKM23}. Sixty years later, Gishboliner, Krivelevich, and Michaeli \cite{GKM23} introduced an {\em oriented discrepancy of Hamilton cycles} in Hamilton undirected graphs. For a Hamilton orcycle $C$, its {\em discrepancy} is $|\sigma^+(C)-\sigma^-(C)|$. 
For a Hamilton graph $G$, its {\em oriented discrepancy} ${\rm disc}(G)$ is the maximum integer $d$ such that every orientation of $G$ has a Hamilton orcycle with discrepancy at least $d$. 
Clearly, to compute  ${\rm disc}(G)$ it is enough to consider Hamilton orcycles $C$ with $\sigma^+(C)\ge \sigma^-(C).$ Since $\sigma^+(C) -  \sigma^-(C)=2\sigma^+(C) - n$, where $n$ is the order of $G,$ the problem of finding a Hamilton orcycle of maximum discrepancy in an oriented graph is equivalent to the problem of finding a Hamilton orcycle with the maximum number of forward arcs (or minimum number of backward arcs). 

Gishboliner et al. \cite{GKM23} conjectured the following strengthening of Dirac's theorem on Hamilton undirected graphs: If the minimum degree $\delta$ of an $n$-vertex graph $G$ is greater or equal to $n/2$, then every orientation of $G$ 
has a Hamilton orcycle with at least $\delta$ forward arcs. Freschi and Lo \cite{FL24} proved this conjecture and asked whether there is a strengthening of Ore's theorem on Hamilton undirected graphs to oriented graphs. In  \cite{Ai+26Ore}  the authors came up with two conjectures for such a strengthening and provided some support for each of them. 

A digraph $D$ is {\em semicomplete multipartite} if its vertex set can be partitioned into at least two non-empty subsets called {\em parts} such that every arc of $D$ connects vertices from different parts and every pair of vertices from different parts is connected by at least one arc.  A semicomplete multipartite digraph is {\em semicomplete} if each part has only one vertex. A digraph is {\em locally semicomplete} if the out-neighborhood and in-neighborhood of each vertex induce semicomplete digraphs.

In \cite{Guo+SMDs} the authors studied the problem of computing the discrepancy of Hamiltonian orcycles and or paths for some  generalizations of tournaments. They proved that both problems are polynomial-time solvable for semicomplete digraphs and for locally semicomplete digraphs. They also proved that for some other generalizations of tournaments,  both problems are {\sf NP}-hard. 

The {\em independence number} $\alpha(D)$ of a digraph $D$ is the maximum integer $k$ such that $D$ has $k$ vertices without arcs between any pair of them.  Clearly, $\alpha(D)=1$ if and only if $D$ is {\em semicomplete}. 
The {\em underlying undirected multigraph} $UG(D)$ of $D$ is the multigraph obtained from $D$ by replacing each arc $xy$ of $D$ by an edge between $x$ and $y$. We call $D$ {\em connected} ({\em $\ell$-connected}, resp.) if $UG(D)$ is connected ($\ell$-connected, resp.). A digraph $D$ is {\em strongly connected} (or, just {\em strong}) if there is an $(x,y)$-dipath and a $(y,x)$-dipath for every pair $x,y$ of vertices of $D$. 

The following theorem of R{\'e}dei is well-known. 
\begin{theorem} \label{HamPathTour}\cite{redei_1934_hamilton_paths}
Every semicomplete digraph has a Hamilton dipath. 
\end{theorem}
Hence, every $n$-vertex semicomplete digraph $D$ has a Hamilton orpath with $0$ backward arcs and a Hamilton orcycle with at most $1$ backward arc.
Also, by Camion's theorem \cite{Camion59}, there is a Hamilton orcycle with $0$ backward arcs in $D$ if and only if $D$ is strongly connected. 

In this paper, we study bounds on the minimum number of backward arcs in a Hamilton orcycle (a Hamilton orpath, respectively) of digraphs with independence number 2, a class of digraphs which is much richer than the class of semicomplete digraphs.
Such digraphs were studied in several papers including \cite{BangJensen2006+,Bondy1995ShortProofChenManalastas,chen_manalastas_1983,FrickKatrenic2008,FrickEtAl2008,Havet2004,AardtEtAl2010}.

Our main results are as follows:

\begin{theorem} \label{mainHP}
(i) If $D$ is a connected digraph with $\alpha(D) \leq 2$, then $D$ contains a Hamilton orpath with at most two backward arcs. (ii) There is a connected digraph $D$ with $\alpha(D) \leq 2$ which has no Hamilton orpath with less than two backward arcs. 
\end{theorem}

\begin{theorem} \label{mainHC}
(i) If $D$ is a 2-connected digraph with $\alpha(D) \leq 2$, then $D$ contains a Hamilton orcycle with at most five backward arcs. (ii) There is a 2-connected digraph $D$ with $\alpha(D) \leq 2$ which has no Hamilton orcycle with less than four backward arcs. 
\end{theorem}

We have the following remarks on Theorems \ref{mainHP} and \ref{mainHC}.

\begin{remark}
The fact that Theorems \ref{mainHP} and \ref{mainHC} give a constant  upper-bound for the number of backward arcs  is somewhat surprising. For example, in another class of generalizations of tournaments, semicomplete multipartite digraphs, there is no such a bound: consider a semicomplete bipartite digraph with two parts of the same size where all arcs are oriented from one part to the other one. 
\end{remark}

\begin{remark}
Connectivity of $UG(D)$ in Theorem \ref{mainHP} and 2-connectivity of $UG(D)$ in Theorem \ref{mainHC} are necessary and sufficient conditions for $D$ to have a Hamilton orpath and a Hamilton orcycle, respectively. Clearly, connectivity is a necessary condition for $D$ to have a Hamilton orpath. It is sufficient as every strongly connected digraph $D$ with $\alpha(D)\le 2$ has a Hamilton dipath, see Theorem \ref{HamPath} below. Clearly, 2-connectivity is necessary for an undirected graph to have a Hamilton cycle. Chv{\'a}tal  and Erd{\H{o}}s \cite{ChEr72} proved that 
an undirected graph with independence number at most its connectivity has a Hamilton cycle. Thus, a 2-connected digraph $D$ with independence number 2 has a Hamilton orcycle. 
\end{remark}

\begin{remark} 
In Theorem \ref{mainHC} there is a small gap of $1$ between the upper bound and the lower bound and we believe the lower bound is sharp.
\end{remark}


\paragraph{Additional Notation}
For $X\subseteq V$ and $Y\subseteq V$, we say that an arc $xy$ is an $(X,Y)$-{\em arc} if $x\in X$ and $y\in Y$ and a dipath $P=x_1x_2\dots x_p$ is an $(X,Y)$-{\em dipath} if $x_1\in X$ and $x_p\in Y.$ For an orpath $P=x_1x_2\dots x_p$ and $1\le i\le j\le p$, $P[x_i,x_j]=x_ix_{i+1}\dots x_j$. A similar notation can be used for orcycles, but then the indexes are taken modulo $p.$ 
We write $X \Rightarrow Y$ in $D$ if for every $x \in X$ and every $y \in Y$ we have $xy \in A$.



\2


To prove Theorems \ref{mainHP} and \ref{mainHC} , we will use the following well-known results.


\begin{theorem} \label{HamPath}\cite{chen_manalastas_1983} (Chen-Manalastas)
 If $D$ is a strong digraph with $\alpha(D) \leq 2$ then $D$ contains a Hamilton dipath.
\end{theorem}


\begin{theorem}\cite{PathPartition} (Gallai-Milgram) \label{PathPartition} For every integer $k\ge 1$,
each digraph $D$ with $\alpha(D) \le k$ has at most $k$ pairwise disjoint dipaths covering all vertices of $D$.
\end{theorem}

\paragraph{Paper organization} We prove Theorem \ref{mainHP}  in Section \ref{sec:HP}.  We prove part (i) of Theorem \ref{mainHC}  in Section \ref{sec:HC} and Part (ii) in Section \ref{sec:OP} (in Theorem \ref{lowbound4k} for $k=2$) where we briefly consider digraphs of higher independence number and state some open problems.

\section{Proof of Theorem \ref{mainHP}}\label{sec:HP}

To prove Part (ii) of Theorem \ref{mainHP} consider the following simple example.
Let $T_a$ and $T_b$ be the  two vertex-disjoint transitive tournaments with
\[
V(T_a)=\{a_1,a_2,\dots,a_k\},\quad 
V(T_b)=\{b_1,b_2,\dots,b_m\},
\]
where $k,m\ge 3$, and 
\[
A(T_a)=\{a_ia_j:\ \text{ for all } 1\le i<j\le k\}, \quad 
A(T_b)=\{b_ib_j:\ \text{ for all } 1\le i<j\le m\}.
\]
Define a digraph $D$ by taking the disjoint union of $T_A$ and $T_B$ and adding a single arc from $a_1$ to $b_2$. See Figure \ref{fig: tight eg for HOP} when $k=m=3$.

Clearly, $\alpha(D)=2$ as any independent set can contain at most one vertex from $T_a$ and at most one vertex from $T_b$. Let $P$ be a Hamilton orpath of $D$. If the arc $a_1b_2$ is backward in $P$, then $P$ contains another backward arc in $T_b$, and if the arc $a_1b_2$ is forward of $P$, then $P$ contains at least two backward arcs, one in $T_a$ and another in $T_b$.
Consequently, any Hamilton orpath of $D$ contains at least two backward arcs. Moreover, the Hamilton orpath $a_2a_3 \dots a_k a_1 b_2 b_3 \dots b_m b_1$ of $D$ contains exactly two backward arcs.

\begin{figure}[h]
 \centering  
\begin{tikzpicture}[node distance=1.5cm,scale=0.9]

    \draw (0,0) node[fill=black,minimum size =4pt,circle,inner sep=1pt,label=west:$a_3$] (a3) {};
    \draw (0,1.5) node[fill=black,minimum size =4pt,circle,inner sep=1pt,label=west:$a_2$] (a2) {};
    \draw (0,3) node[fill=black,minimum size =4pt,circle,inner sep=1pt,label=west:$a_1$] (a1) {};


    \draw (3,0) node[fill=black,minimum size =4pt,circle,inner sep=1pt,label=west:$b_3$] (b3) {};
    \draw (3,1.5) node[fill=black,minimum size =4pt,circle,inner sep=1pt,label=west:$b_2$] (b2) {};
    \draw (3,3) node[fill=black,minimum size =4pt,circle,inner sep=1pt,label=west:$b_1$] (b1) {};

\draw[arrows={-Stealth[reversed, reversed]}] (a1) -- (a2);
\draw[arrows={-Stealth[reversed, reversed]}] (a2) -- (a3);

\draw[arrows={-Stealth[reversed, reversed]}] (a1) to[out=315, in=40] (a3);

\draw[arrows={-Stealth[reversed, reversed]}] (b1) -- (b2);
\draw[arrows={-Stealth[reversed, reversed]}] (b2) -- (b3);
\draw[arrows={-Stealth[reversed, reversed]}] (b1) to[out=315, in=40] (b3);

\draw[arrows={-Stealth[reversed, reversed]}, thick] (a1) -- (b2);

\end{tikzpicture}
\caption{An example when $k=m=3$}
\label{fig: tight eg for HOP}
\end{figure}
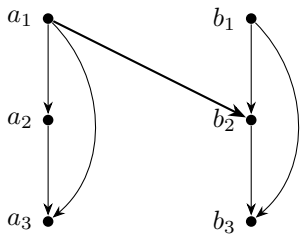

We first prove Part (i) of Theorem \ref{mainHP} for a special case. 

\begin{lemma} \label{TwoTournaments}
 Let $D$ be a connected digraph with $\alpha(D) \leq 2$ and such that $V(D) = V_1 \cup V_2$, $V_1\cap V_2=\emptyset$ and $D[V_i]$ is semicomplete for $i \in \{1,2\}$.
Then $D$ contains a Hamilton orpath with at most two backward arcs.
\end{lemma}

\begin{proof}
 Let $D$ be a connected digraph with $\alpha(D) \leq 2$ and such that $V(D) = V_1 \cup V_2$, $V_1\cap V_2=\emptyset$ and $D[V_i]$ is semicomplete for $i \in \{1,2\}$.
Let $P=p_1 p_2 p_3 \ldots p_s$ be a Hamilton dipath in $D[V_1]$ and let
$Q=q_1 q_2 q_3 \ldots q_t$ be a Hamilton dipath in $D[V_2]$ (both Hamilton dipaths exist by Theorem~\ref{HamPathTour}).
As $D$ is connected we may assume, without loss of generality, that there exists an arc $p_i q_j$ in $D$.
Observe that 
\(
R=p_{i+1} p_{i+2} \ldots p_s p_1 p_2 \ldots p_i q_j q_{j+1} \ldots q_t q_1 q_2 \ldots q_{j-1}
\)
is a Hamilton orpath, where only arcs $p_s p_1$ and $q_t q_1$ may be backward.
\end{proof}

In the rest of the section, we will prove the general case of Part (i) of Theorem \ref{mainHP}.


\begin{lemma} \label{AddArc}
 Let $D$ be a connected digraph with $\alpha(D) \leq 2$ and let $A' \subseteq A(D)$.
If $D$ becomes strong by adding the reverse arc of the arcs in $A'$, then 
 $D$ contains an orpath with at most $|A'|$ backward arcs.
\end{lemma}

\begin{proof}
 Let $A' \subseteq A(D)$ and let $D'$ be obtained 
from $D$ by adding the reverse arcs of the arcs in $A'$.
If $D'$ is strong, then it contains a Hamilton dipath, $P$, by Theorem~\ref{HamPath}.
All backward arcs of the orpath $P$ in $D$  lie in $A'$, completing the proof.
\end{proof}

\begin{lemma}\label{adding-arcs-for-strong}
If $D$ is a connected digraph such that $\alpha(D)\le 2$, then either $V(D)=V_1\cup V_2,$  $V_1\cap V_2=\emptyset$ and $D[V_i]$ is semicomplete for $i \in \{1,2\}$, or we can add, to $D$, at most two arcs, which are reverses of existing arcs, such that $D$ becomes strong. 
\end{lemma}
\begin{proof}
Let $D$ be a connected digraph with $\alpha(D) \leq 2$. We may assume that $D$ is non-strong. 
Let $I_1,I_2,\ldots,I_s$ be the initial strong components of $D$ and let 
$T_1,T_2,\ldots, T_t$ be the terminal strong components of $D$. 
As there are no arcs between distinct initial components we note that $1 \leq s \leq 2$, as $\alpha(D) \leq 2$.
Analogously, we also note that $1 \leq t \leq 2$.  We now consider the following three cases which exhaust all possibilities.



\2

{\bf Case 1: $1=s=t$.} 
In this case $I_1$ is the unique initial strong component and $T_1$ is the unique terminal strong component and $I_1 \not= T_1$. 
If there is an arc between $I_1$ and $T_1$ then it must be an $(I_1,T_1)$-arc, $a$. Adding the reverse arc of $a$ gives us a strong digraph.
 Thus, we may assume that there are no arcs between $I_1$ and $T_1$. 

Let $X$ contain all vertices of $I_1$ as well as all vertices in $V(D) \setminus (I_1 \cup T_1)$ which have no arc to $T_1$
and let $Y = V(D) \setminus X$.
Note that no vertex in $X$ has an arc to $T_1$ which implies that $\alpha(D[X])=1$, as if there would be two non-adjacent vertices in $X$
then together with a vertex from $T_1$ they would form an independent set of size 3 in $D$.

Note that $V(T_1) \subseteq Y$. For the sake of contradiction assume that $y \in Y$ and there exists
an $(I_1,y)$-arc, $a_I$, in $D$. This implies that $y \not\in V(T_1)$ and since $y \not\in X$ there exists a
$(y,T_1)$-arc, $a_T$, in $D$. Note that $D$ becomes strong when we add the reverse arcs of $a_I$ and $a_T$. 
Thus, we may assume that there is no $(I_1,y)$-arc for any $y \in Y$.

This implies that $\alpha(D[Y])=1$, 
as if there would be two non-adjacent vertices in $Y$,
then together with a vertex from $I_1$ they would form an independent set of size 3 in $D$.

Now $V(D)$ can be partitioned into two semicomplete digraphs, induced by $X$ and $Y$ and we are done.

\2

{\bf Case 2: $\{s,t\}=\{1,2\}$.} We can without loss of generality, assume that $s=2$ and $t=1$ (otherwise we can reverse all arcs).
So $I_1$ and $I_2$ are the two initial strong components of $D$ and $T_1$ is the unique terminal strong component of $D$.
As $I_1$ and $I_2$ are both initial strong components of $D$ there is no arc between them.

If there exists an $(I_1,T_1)$-arc, $a_1$, and an $(I_2,T_1)$-arc, $a_2$, in $D$, then adding the reverse arc of $a_1$ and the reverse arc of $a_2$
gives us a strong digraph. Thus, without loss of generality we may assume that there is no $(I_2,T_1)$-arc
in $D$. This implies that $I_1 \Rightarrow T_1$, as if $uv \not\in A(D)$ for some $u \in I_1$ and $v \in T_1$, then $u$ and $v$ and any vertex
in $I_2$ would form an independent set of size three in $D$.

Let $X$ contain all vertices that can be reached by a dipath starting in $I_1$. Note that $I_1 \subseteq X$. Let $Y = V(D) \setminus X$.
We now consider the following two subcases.

\2

{\bf Subcase 2a: there exists an $(I_2,X)$-arc.} 
Let $a'$ be any $(I_1,T_1)$-arc in $D$ (which exists by the above) and let $a^*$ be any $(I_2,X)$-arc.
Let $D'$ be obtained from $D$ by adding the reverse of $a'$ and the reverse of $a^*$.
Clearly every vertex in $D'$ can reach every vertex in $T_1$ as this is the case for $D$ (as $T_1$ is the unique terminal strong component).
Every vertex in $T_1$ can reach every vertex in $I_1$ in $D'$ (as both $T_1$ and $I_1$ are strongly connected and the reverse of $a'$ is a
$(T_1,I_1)$-arc in $D'$). Thus, if $a^*=uv$, then every vertex in $D'$ can reach $v$ (as $v \in X$ so $v$ can be reached from $I_1$).
Using the reverse of $a^*$ we note that every vertex in $T_1$ can reach both $I_1$ and $I_2$ in $D'$, so $D'$ is strong.

\2

{\bf Subcase 2b: there is no $(I_2,X)$-arc.} 
As no vertex in $X$ has an arc into it from $I_2$ we must have $\alpha(D[X])=1$, as if there would be two non-adjacent vertices in $X$
then together with a vertex from $I_2$ they would form an independent set of size 3 in $D$.

As no vertex in $Y$ has an arc into it from $I_1$ (as otherwise it would belong to $X$ and not $Y$) we note that
$\alpha(D[Y])=1$ (if there would be two non-adjacent vertices in $Y$
then together with a vertex from $I_1$ they would form an independent set of size 3 in $D$).

Now $V(D)$ can be partitioned into two semicomplete digraphs, induced by $X$ and $Y$ and we are done.

\2

{\bf Case 3: $s=t=2$.}
In this case $I_1$ and $I_2$ are the two initial strong components of $D$ and $T_1$ and $T_2$ are the two terminal strong component of $D$.
As $I_1$ and $I_2$ are both initial strong components of $D$ there is no arc between them and
analogously there are is no arc between $T_1$ and $T_2$.
We will now prove the following claims.

\2

{\bf Claim~3.A:} {\em We may assume that there exists an $(I_1,T_2)$-dipath and an $(I_2,T_2)$-dipath in $D$.}

\2

{\bf Proof of Claim~3.A:} Let $X_i$ contain all vertices that can be reached by a dipath starting in $I_i$ for $i=1,2$.
Note that $V(D) = X_1 \cup X_2$ as $I_1$ and $I_2$ are the only initial strong components in $D$.
If $X_1 \cap X_2 = \emptyset$, then $D$ would not be connected (there would be no arc between $X_1$ and $X_2$),
so we must have $X_1 \cap X_2 \not= \emptyset$. Let $x \in X_1 \cap X_2$ be arbitrary. 
As $T_1$ and $T_2$ are the only terminal strong components in $D$ we must have an $(x,T_j)$-dipath for some $j \in \{1,2\}$.
We can name $T_1$ and $T_2$ such that $j=2$. Now we note that there exists an $(I_1,T_2)$-dipath and an $(I_2,T_2)$-dipath in $D$ 
(going through $x$), as desired.

\2

{\bf Claim~3.B:} {\em We may assume that there exists an $(I_1,T_1)$-arc and an $(I_2,T_2)$-arc in $D$.}

\2

{\bf Proof of Claim~3.B:} If there is no $(I_j,T_2)$-arc in $D$ for any $j \in \{1,2\}$ then we get a contradiction 
to $\alpha(D) \leq 2$, by taking any vertex in $I_1$ and in $I_2$ and in $T_2$. So without loss of generality 
we may assume that there is an $(I_2,T_2)$-arc in $D$ (otherwise change the names of $I_1$ and $I_2$).

Assume that there is no $(I_1,T_1)$-arc in $D$, as otherwise we are done. This implies that $I_2 \Rightarrow T_1$, 
as otherwise we can get a contradiction  to $\alpha(D) \leq 2$ (by taking a vertex in $I_1$, in $I_2$ and in $T_1$).
As there is no $(I_1,T_1)$-arc in $D$ we must have a $I_1 \Rightarrow T_2$ in $D$, as otherwise
we can get a contradiction  to $\alpha(D) \leq 2$ (by taking a vertex in $I_1$, in $T_1$ and in $T_2$). 

We have now shown that $I_2 \Rightarrow T_1$ and $I_1 \Rightarrow T_2$, so renaming $I_1$ and $I_2$ we get the desired
$(I_1,T_1)$-arc and an $(I_2,T_2)$-arc in $D$.

\2

{\bf Definition:} By Claim~3.B, let $a_1$ be an $(I_1,T_1)$-arc in $D$ and let $a_2$ be an $(I_2,T_2)$-arc in $D$.
Let $D^*$ be obtained from $D$ by adding the reverse of $a_1$ and the reverse of $a_2$.
As $I_1$, $I_2$, $T_1$ and $T_2$ are all strongly connected and in $D^*$ there is a $2$-dicycle between $I_1$ and $T_1$
and a $2$-dicycle between $I_2$ and $T_2$ we note that $I_1 \cup T_1$ is a subset of some strong component $C_1$ in $D^*$ 
and $I_2 \cup T_2$ is a subset of some strong component $C_2$ in $D^*$.

\2

{\bf Claim~3.C:} {\em We may assume that $D^*$ contains exactly two strong components, $C_1$ and $C_2$, defined above.}

\2

{\bf Proof of Claim~3.C:} First assume that $C_1 = C_2$. In this case $D^*$ is strong, as every vertex in $D^*$
has a dipath to $T_1 \cup T_2$ and every vertex has a dipath from $I_1 \cup I_2$. Therefore we are done by 
Lemma~\ref{AddArc} in this case.
So we may assume that $C_1 \not= C_2$.

By Claim~3.A there exists an $(I_1,T_2)$-dipath in $D$, which implies that there is an $(C_1,C_2)$-dipath in $D^*$.
For the sake of contradiction assume that $C$ is a strong component in $D^*$ which is different from $C_1$ and $C_2$.
There exists an $(I_i,C)$-dipath and a $(C,T_j)$-dipath in $D^*$ for some $i,j \in \{1,2\}$.
If $i=j$ then $C$ has a dipath from and to $C_i$, a contradiction to $C$ being different from $C_i$. 
If $j=1$ and $i=2$ then there exists a $(C_2,C)$-dipath and a $(C,C_1)$-dipath, which together with the $(C_1,C_2)$-path in $D^*$
implies that $C=C_1=C_2$, a contradiction. So the only option is $i=1$ and $j=2$. 

Let $u \in V(I_2)$ and let $v \in V(T_1)$ and let $w \in V(C)$. As $\{u,v,w\}$ is not independent in $D$ one of the arcs
$uw$, $uv$ or $wv$ mst belong to $D$.  If $uw \in A(D)$ we could have chosen $i=2$, a contradiction. If $wv \in A(D)$
we could have chosen $j=1$, a contradiction. Thus, $uv \in A(D)$. But in this case $uv$ is a $(C_2,C_1)$-arc in $D^*$
which together with the $(C_1,C_2)$-path in $D^*$ implies that $C_1=C_2$, a contradiction.
Thus, $C$, does not exist and the Claim is proved.

\2

{\bf Remainder of the proof.} By Claim~3.C, $C_1$ and $C_2$ are the only two strong components of $D^*$.
By Claim~3.A there exists an $(I_1,T_2)$-dipath in $D$, which implies that there is a $(C_1,C_2)$-arc in $D^*$.
We will now show that $\alpha(C_1)=1$ and $\alpha(C_2)=1$.

For the sake of contradiction assume that $\alpha(C_1)>1$ and let $u,v \in V(C_1)$ such that $u$ and $v$ are non-adjacent in $D$.
Let $w \in V(I_2)$ be arbitrary. We must have $wu \in A(D)$ or $wv \in A(D)$ as otherwise $\alpha(D) \geq 3$.
But this implies that there exists a $(C_2,C_1)$-arc in $D^*$, a contradiction (as there exists a $(C_1,C_2)$-arc in $D^*$).
Thus, $\alpha(C_1)=1$.

Analogously, if $\alpha(C_2)>1$, then considering two non-adjacent vertices in $C_2$ and a vertex in $T_1$ we obtain a 
contradiction. Thus, $\alpha(C_1)=1$ and $\alpha(C_2)=1$ and we are done.
\end{proof}

\section{Proof of Theorem \ref{mainHC}}\label{sec:HC}

Part (ii) of Theorem \ref{mainHC} follows from Theorem \ref{lowbound4k} (for $k=2$) proved in the next section. 
The rest of this section is devoted to the proof of Part (i) of Theorem \ref{mainHC}.


Before proving Theorem \ref{mainHC}, we give two definitions and prove eight lemmas.


\begin{definition}
Let $D$ be a semicomplete digraph and let $x,y \in V(D)$ be arbitrary distinct vertices in $D$.
Let  $b(x,y)$ be the minimum number of backward arcs in any Hamilton $(x,y)$-orpath, and let
$b(y,x)$ be the minimum number of backward arcs in 
any Hamilton $(y,x)$-orpath.  
Furthermore, for a subset $V \in V(D)$ and $x,y \in V$, let $b_V(x,y)$ be the minimum number of backward arcs in any Hamilton $(x,y)$-orpath in $D[V]$, and let $b_V(y,x)$ be the minimum number of backward arcs in any Hamilton $(y,x)$-orpath $D[V]$.
\end{definition}

\begin{definition}
Let $D$ be a digraph. We define $D^*$ as a digraph obtained from $UG(D)$ by substituting every edge with a 2-dicycle. That is $uv \in A(D^*)$ if and only if $uv \in A(D)$ or $vu \in A(D)$.   
\end{definition}

\begin{lemma} \label{PathStrongSemiC}
Let $D$ be a strong semicomplete digraph and let $x,y \in V(D)$ be arbitrary distinct vertices in $D$. The following two statements now hold.

\begin{description}
\item[(a)] If $C$ is a Hamilton dicycle in $D$, then either $b(x,y) \le 1$ or $xy \in A(C)$ (or both).
\item[(b)] $b(x,y) + b(y,x) \le 2$.
\end{description}
\end{lemma}

\begin{proof}
We first prove part~(a). Let $C$ be a Hamilton dicycle in $D$ and assume that $C=v_1v_2v_3 \ldots v_nv_1$ and 
assume without loss of generality that $x = v_1$ and $y=v_r$ where $r \in \{2,3,4, \ldots , n\}$.
We may assume that $r \geq 3$ as otherwise (a) holds.  And that $r<n$ as otherwise $v_1v_2v_3 \ldots v_n$ is a Hamilton $(x,y)$-orpath with no backward arcs. So, $ 2 < r < n$.

Assume that  $v_n v_j \in A(D) $ for some $j \in \{2,3,\ldots, r\}$. 
The following  Hamilton $(v_1,v_r)$-orpath now only has at most one backward arc (possibly the arc $v_{j-1} v_{r+1}$).
\[
v_1 v_2 \ldots v_{j-1} v_{r+1} v_{r+2} \ldots v_n v_j v_{j+1} v_{j+2} \ldots v_r
\]
We may therefore assume that $v_n v_j \not\in A(D)$ for all $j \in \{2,3,\ldots, r \}$.
We may also assume that $v_{r+1} v_j \in A(D)$ for all $j \in \{1,2,\ldots, r-1 \}$ 
as otherwise the following Hamilton $(v_1,v_r)$-orpath has
at most one backward arc, $v_n v_{j+1}$ (as $v_j v_{r+1} \in A(D)$).
\[
v_1 v_2 \ldots v_j v_{r+1} v_{r+2} \ldots v_n v_{j+1} v_{j+2} v_{j+3} \ldots v_r
\]
Let $s$ be the largest value such that there exists an arc $v_s v_j$ in $D$ where $j \in \{2,3,\ldots,r-1\}$.
By the above, we note that $r< s <n$ and  the following Hamilton $(v_1,v_r)$-orpath has
at most one backward arc, $v_n v_{r+1}$ (as $v_{j-1} v_{s+1} \in A(D)$ by the maximality of $s$).
\[
v_1 v_2 \ldots v_{j-1} v_{s+1} v_{s+2} \ldots v_n v_{r+1} v_{r+2} v_{r+3} \ldots v_s v_j v_{j+1} \ldots v_r
\]
This completes the proof of part~(a).

In order to prove part~(b) let $C$ be any hamilton orcycle in $D$. If $x$ and $y$ are not consecutive vertices on $C$ then part~(b) follows from part~(a). So, without loss of generality assume that $C=v_1v_2v_3 \ldots v_n$ is a hamilton orcycle in $D$ 
and $x=v_1$ and $y=v_n$. Then $v_1v_2v_3 \ldots v_n$ is a Hamilton $(x,y)$-orpath in $D$ (with zero backward arcs) and
$v_nv_2v_3v_4 \ldots v_{n-1} v_1$ is a Hamilton $(y,x)$-orpath in $D$ with at most two backward arcs (possibly
$v_nv_2$ and $v_{n-1} v_1$). So in all cases $b(x,y) + b(y,x) \le 2$ as desired.
\end{proof}

\begin{lemma} \label{lem: XYpathSemiC}
Let $D$ be a semicomplete digraph and let $x,y \in V(D)$ be arbitrary distinct vertices. Then 
\begin{description}
\item[(a)] $b(x,y)\le2$.

\item[(b)] $b(x,y)\le 1$ or $b(y,x)\le 1$.
\end{description}
\end{lemma}

\begin{proof}
We first prove part~(a).
If $|V(D)| \le 3$ then clearly the lemma holds, so assume that  $|V(D)| \geq 4$.
Let $D' = D - y$ and note that adding at most one arc to $D'$, we can make $D'$ strong.  Therefore, there exists a Hamilton orcycle in $D'$ with 
at most one backward arc. Deleting the arc between $x$'s predecessor, $x^-$, and $x$ and adding the arc between $x^-$ and $y$ instead gives us the desired orpath in $D$.
This completes the proof of part~(a).

Now we prove part~(b).
First, suppose that $D$ is strong. Let $C$ be a Hamilton dicycle in $D$ and let $x,y \in V(D)$ be arbitrary distinct vertices.

We denote $x^+,y^+$ the successors of $x,y$ on $C$, and $x^{++},y^{++}$ the successors of $x^+,y^+$ on $C$.

\vspace{0.5em}
\textbf{Case 1:} $x^+=y$ or $x=y^+$.

There is a Hamilton dipath between $x$ and $y$ in $D$.

\vspace{0.5em}
\textbf{Case 2:} $x^{++}=y$ or $x=y^{++}$, and $x^+\ne y$ and $x\ne y^+$.

If $x^{++}=y$ and $x=y^{++}$, then $|V(D)|=4$ and the length of any Hamilton orpath between $x$ and $y$ is $3$, so this orpath or its reverse is the desired orpath.

If not, without loss of generality, $x^{++}=y$ and $x\ne y^{++}$, and then  $y$'s predecessor on $C$, $y^-$ is $x^+$. Then one of the following two Hamilton orpaths is the desired orpath, depending on the arc between $x^+$ and $y^+$.

\[
P_1=x x^+ C[y^+, x^-] y,~P_2=y y^+ x^+ C[y^{++}, x]
\]

\vspace{0.5em}
\textbf{Case 3:} $x^{++}\ne y$, $x\ne y^{++}$, $x^+\ne y$, and $x\ne y^+$.

Then one of the following two Hamilton orpaths is the desired orpath, depending on the arc between $x^+$ and $y^+$.
\[
P_1=x x^+ C[y^+, x^-] C[x^{++}, y], ~P_2=y y^+ C[x^+, y^-]C[y^{++}, x]
.\]
This completes the case when $D$ is strong.

Now consider the case when that $D$ is not strong.
Decompose $D$ into strong components $V_1,\dots,V_l$ with all arcs pointing from $V_i$ to $V_j$ whenever $i<j$ ($l>1$). All strong semicomplete digraphs contain a Hamilton dicycle, and hence contain a Hamilton dipath starting or ending at any vertex (not necessarily at the same time). We denote $P_i$ to be an arbitrary Hamilton dipath in $V_i$, and for a vertex $v$, we denote $S_v$ ($E_v$, respectively) to be a Hamilton dipath in the strong component containing $v$ starting at $v$ (ending at $v$, respectively).

Let $x \in V_a$ and $y \in V_b$.
Without loss of generality, we may assume that $a \le b$.

\textbf{Case 1:} $a<b$.
In this case, the desired Hamilton $(x,y)$-orpath is the following:
\[
S_xP_{a+1}\dots P_{b-1}P_{b+1}\dots P_lP_1 \dots P_{a-1}E_y
\]
The only backward arc is the one going from $P_l$ to $P_1$.

\textbf{Case 2:} $a=b$. 
Without loss of generality, we may assume that $b(x,y)\le 1$ in the strong component $V_a$.
Let $u_1u_2\dots u_s$, with $u_1=x, u_s=y$, be the Hamilton $(x,y)$-orpath with at most one backward arc in $V_a$.
Say the backward arc (if it exists) is between $u_i$ and $u_{i+1}$ for $i \in \{1,2,\dots s-1 \}$. 
Then the desired Hamilton $(x,y)$-orpath in $D$ is the following:
\[
xu_2\dots u_{i}P_{a+1}\dots P_lP_1 \dots P_{a-1}u_{i+1}\dots u_{s-1}y.
\]
Again, the only backward arc is the one going from $P_l$ to $P_1$.
\end{proof}   

\begin{lemma}\label{lem: TwoSemiC}
Let $D$ be a digraph such that $X$ and $Y$ are disjoint subsets of $V(D)$ and $D[X]$ and $D[Y]$ induce semicomplete digraphs.
Let $P$ and $Q$ be $(X,Y)$-paths in $UG[D]$  such that $P$ and $Q$ are vertex disjoint and 
$V(D) = X \cup Y \cup V(P) \cup V(Q)$. 
Then $D$ contains a Hamilton orcycle with at most $2 + (|V(P)|+|V(Q)|)/2$ backward arcs.
\end{lemma} 

\begin{proof}
Let $D$ be a digraph such that $X$ and $Y$ are disjoint subsets of $V(D)$ and $D[X]$ and $D[Y]$ induce semicomplete digraphs.
Let $P$ and $Q$ be paths as described in the statement of the lemma. 
We may assume that $P$ and $Q$ are minimal in the sense that $|V(P) \cap X|=|V(P) \cap Y|=|V(Q) \cap X|=|V(Q) \cap Y|=1$.
Let $P$ be a $(p,p')$-path and let $Q$ be a $(q,q')$-path such that $p,q\in X$ and $p',q'\in Y$.

By Lemma \ref{lem: XYpathSemiC}, $D[X]$ ($D[Y]$, respectively) contains two Hamilton orpaths $H_1$ ($H_1'$, respectively) and $H_2$ ($H_2'$, respectively) with same end-vertices $p,q$ ($p',q'$, respectively) such that $H_1$ ($H_1'$, respectively) starts at $p$ ($p'$, respectively) , $H_2$ ($H_2'$, respectively) starts at $q$ ($q'$, respectively) and the sum of the numbers of backward arcs in $H_1$ and $H_2$ ($H_1'$ and $H_2'$, respectively) is at most 3. 

Let $C_1$ be the Hamilton orcycle in $D$ containing $H_1$, $Q$, $H_2'$ and $P$. Let $C_2$ be the Hamilton orcycle in $D$ containing $H_2$, $P$, $H_1'$ and $Q$. Since the sum of the numbers of backward arcs in $C_1$ and $C_2$ is at most $6+ |V(P)|-1+|V(Q)|-1$, at least one of $C_1$ and $C_2$ has at most  $2 + (|V(P)|+|V(Q)|)/2$ backward arcs
\end{proof}

\begin{corollary}\label{cor: TwoSemiC}
Let $D$ be a 2-connected digraph such that $X$ and $Y$ are disjoint subsets of $V(D)$, $X \cup Y = V(D)$, and $D[X]$ and $D[Y]$ induce semicomplete digraphs. Then $D$ contains a Hamilton orcycle with at most $4$ backward arcs.  
\end{corollary}

\begin{proof}
Since $D$ is 2-connected, there exist $x_1 \neq x_2 \in X$ and $y_1 \neq y_2 \in Y$ such that $x_1y_1, x_2y_2 \in A(D^*)$. Applying Lemma \ref{lem: TwoSemiC} with $X=X, Y=Y, P=x_1y_1, Q=x_2y_2$ yields a Hamilton orcycle with at most $2+(2+2)/2=4$ backward arcs.
\end{proof}

\begin{lemma}\label{lem: FullUnionNonStrong}
Let $D$ be a non-strong semicomplete digraph. Let $X$ and $Y$ be non-empty subsets of $V(D)$ such that $X \cup Y = V(D)$. Then there exist $x \in X$ and $y \in Y$ such that $b(x,y) \le 1$ and $b(y,x) \le 1$.
\end{lemma}

\begin{proof}
Decompose $D$ into strong components $V_1,\dots,V_l$ with all arcs pointing from $V_i$ to $V_j$ whenever $i<j$. All strong semicomplete digraphs contain a Hamilton dicycle, and hence contain a Hamilton dipath starting or ending at any vertex (not necessarily at the same time). We denote $P_i$ to be an arbitrary Hamilton dipath in $V_i$, and for a vertex $v$, we denote $S_v$ ($E_v$, respectively) to be a Hamilton dipath in the strong component containing $v$ starting at $v$ (ending at $v$, respectively).

Without loss of generality, assume $Y \not\subset V_1$. If $X \cap V_1 \neq \emptyset$, we pick $x \in X \cap V_1$ and $y \in Y \setminus V_1$. If $X \cap V_1 = \emptyset$, we pick $y\in Y \cap V_1$ and $x \in X \setminus V_1$. In either case, we have $x \in X, y\in Y$ such that $x \in V_i, y \in V_j$ for some $i \neq j$.
Without loss of generality, we may assume that $i<j$. Then there must be some $k \in \{i,i+1,\dots j-1\}$ such that $X\cap V_k\ne \emptyset$ and $Y\cap V_{k+1}\ne \emptyset$, and we finally choose new $x\in X\cap V_k$ and new $y\in Y\cap V_{k+1}$.

The Hamilton orpath 
\[
S_{x}P_{k+2}\dots P_l P_1 \dots P_{k-1}E_y
\]
yields $b(x,y) \le 1$, and the Hamilton orpath
\[
S_{y}P_{k+2}\dots P_l P_1 \dots P_{k-1}E_x
\]
yields $b(y,x) \le 1$.
\end{proof}

\begin{lemma}\label{lem: FullUnion}
Let $D$ be a semicomplete digraph. Let $X$ and $Y$ be non-empty subsets of $V(D)$ such that $X \cup Y = V(D)$. Then there exist $x \in X$ and $y \in Y$ such that $b(x,y) \le 1$ and $b(x,y)+b(y,x) \le 2$.
\end{lemma}

\begin{proof}
If $D$ is not strong, we are done by Lemma \ref{lem: FullUnionNonStrong}.
If $D$ is strong, there must be $x \in X$ and $y \in Y$ such that $yx \in A(C)$ for some Hamilton dicycle $C$, and then $b(x,y) = 0$ and by Lemma \ref{PathStrongSemiC}~part (b), $b(x,y)+b(y,x) \le 2$.
\end{proof}

The following lemma was proved, as Lemma 2, in \cite{chen_manalastas_1983}.

\begin{lemma}\label{lem: PathCombine}
Let $P=p_1 \dots p_s$ and $Q=q_1\dots q_t$ be two dipaths in a directed graph $D$, and let there be an arc between $p_i$ and $q_j$ for every $p_i\in V(P)$ and $q_j\in V(Q)$.
Then there is a dipath $R=r_1 \dots r_{s+t}$ in $D[V(P) \cup V(Q)]$ such that $r_1\in \{p_1,q_1\}$ and $r_{s+t}=\{p_s,q_t\}$.   
\end{lemma}

Denote $N_D(v)=N^+_D(v) \cup N^-_D(v)$. For a subset $X \subset V(D)$, we may write $N_X(v) := N_D(v) \cap X$ when $D$ is clear from context.

\begin{lemma}\label{DipathPlusSemi}
Let $D$ be a digraph with $\alpha(D)\le 2$, and let $X$ and $Y$ be non-empty disjoint subsets of $V(D)$ such that $D[X]$ is semi-complete, $D[Y]$ is connected, and $X \cup Y=V(D)$. Let $P=y_1y_2\dots y_k$ be a Hamilton orpath in $D[Y]$ with at most $2$ backward arcs (exists by  Theorem \ref{mainHP}). 
If, in addition, one of the following conditions holds:
\begin{description}
\item[(a)] $y_1y_k \notin A(D^*)$, and either $X$ is a semicomplete digraph of maximum size in $D$, or $|N_X(y_1)|,|N_X(y_k)| \ge 1$.

\item[(b)] $P$ has at most $1$ backward arc, and there exist $x \neq x' \in X $ such that $y_1x, y_kx' \in A(D^*)$.

\item[(c)] $D[Y]$ contained a Hamilton orcycle $C=y_1\dots y_ky_1$ with at most $2$ backward arcs, and there exist $i~\in \{1,2,\dots,k\}$, $x_1\in N_X(y_{i-1})$, $x_2\in N_X(y_i)$ , and $x_3\in N_X(y_{i+1})$ such that $x_2 \notin \{x_1,x_3\}$.

\item[(d)] $|X| \ge 4$, every vertex in $Y$ has at least $1$ neighbour in $X$, and $P$ has at most $1$ backward arc.
\end{description}
Then $D$ contains a Hamilton orcycle with at most $5$ backward arcs.
\end{lemma}

\begin{proof}
For $x \neq x' \in X$, we write $Q[x,x']$ for any Hamilton $(x,x')$-orpath in $D[X]$ with the minimum number of backward arcs (among all Hamilton $(x,x')$-orpath in $D[X]$).

\textbf{(a):} Since $\alpha(D) \le 2$, we have $N_X(y_1) \cup N_X(y_k)=X$. If $X$ is a semicomplete digraph of maximum size in $D$ then $|N_X(y_1)|,|N_X(y_k)| \le |X|-1$, and hence $|N_X(y_1)|,|N_X(y_k)| \ge 1$. So in both cases we may assume that $|N_X(y_1)|,|N_X(y_k)| \ge 1$.

By Lemma \ref{lem: FullUnion}, there exist $x_1 \in N_X(y_1)$ and $x_2 \in N_X(y_k)$ such that $Q[x_2,x_1]$ has at most one backward arc.
The orcycle 
\[
PQ[x_2,x_1]y_1
\]
is a Hamilton orcycle in $D$ with at most 5 backward arcs, with two possible backward arcs in $P$, one in $Q[x_2,x_1]$, and two from $y_kx_2$ and $x_1y_1$.

\textbf{(b):} By Lemma \ref{lem: XYpathSemiC} (a), $Q[x',x]$ has at most $2$ backward arcs. The orcycle 
\[
PQ[x',x]y_1
\]
is a Hamilton orcycle in $D$ with at most 5 backward arcs, with two possible backward arcs in $Q[x',x]$, one in $P$, and two from $y_kx'$ and $xy_1$.

\textbf{(c):} If $x_2y_i \in A(D)$, then 
\[
y_i\dots y_k y_1 \dots y_{i-1}Q[x_1,x_2]y_i
\]
is a Hamilton orcycle in $D$ with at most $5$ backward arcs, with $2$ possible backward arcs in $Q[x_1,x_2]$, $2$ in $y_i\dots y_k y_1 \dots y_{i-1}$, and $y_{i-1}x_1$. Conversely, if $y_ix_2 \in A(D)$, then 
\[
y_{i+1}\dots y_k y_1 \dots y_iQ[x_2,x_3]y_{i+1}
\]
is a Hamilton orcycle in $D$ with at most $5$ backward arcs, with $2$ possible backward arcs in $Q[x_2,x_3]$, $2$ in $y_{i+1}\dots y_k y_1 \dots y_i$, and $x_3y_{i+1}$.

\textbf{(d):} We may assume that $y_1y_k \in A(D^*)$, as otherwise we are done by part~(a). Thus, we now have a Hamilton orcycle in $D[Y]$ with at most 2 backward arcs. We may also assume $|N_X(y_1)| = |N_X(y_k)|=1$ as otherwise we are done by part~(b). 

If there exists $i \in \{2,\dots,k-1\}$ such that $y_1y_i \notin A(D^*)$, then $\alpha(D) \le 2$ implies that $N_X(y_1)\cup N_X(y_i)=X$, which implies that $|N_X(y_i)|\ge |X|-1 \ge 3$. 
Pick any $x_1 \in N_X(y_{i-1})$ and any $x_3 \in N_X(y_{i+1})$. Since $|N_X(y_i)| \ge 3$, we can pick $x_2 \in N_X(y_i) \setminus \{x_1,x_3\}$. Hence, $x_2 \notin \{x_1,x_3\}$ and we are done by part~(c). 

Hence, we may assume that $y_1y_i \in A(D^*)$ for all $i \in \{2,\dots,k\}$. Analogously, we may also assume that $y_ky_i \in A(D^*)$ for all $i \in \{1,\dots,k-1\}$. We may further assume that $Y$ is not semicomplete, as otherwise we are done by Corollary \ref{cor: TwoSemiC}. 
Let $i$ be the minimum number such that there exists $j \in \{i+1,\dots,k-1\}$ such that $y_iy_j \notin A(D^*)$. Note that $i>1$ and $j \neq i+1$ as $y_iy_{i+1} \in A(D^*)$. By the choice of $i$, we have $y_{i-1}y_{j+1} \in A(D^*)$

Consider 
\[
P'=y_j y_{j+1}\dots y_ky_{i+1}y_{i+2}\dots y_{j-1}y_1y_2\dots y_i, \enspace
P''=y_iy_{i+1}\dots y_{j-1}y_1y_2 \dots y_{i-1}y_{j+1}y_{j+2} \dots y_ky_j.
\]
We can see that $P'$ is a Hamilton orpath in $D[Y]$ with at most $3$ backwards arcs ($y_ky_{i+1}, y_{j-1}y_1$, and one backward arc in $P$), with no arc between the $y_j$ and $y_i$, and $P''$ is a Hamilton orpath in $D[Y]$ with at most $4$ backwards arcs ($y_{j-1}y_1, y_{i-1}y_{j+1},y_ky_j$, and one backward arc in $P$), with no arc between the $y_i$ and $y_j$.

Since $\alpha(D) \ge 2$, we have $N_X(y_i) \cup N_X(y_j) = X $. As $|N_X(y_i)|,|N_X(y_j)| \ge 1$, by Lemma \ref{PathStrongSemiC} part~(b) and Lemma \ref{lem: FullUnionNonStrong}, there exist $x_1 \in N_X(y_i)$ and $x_2 \in N_X(y_j)$ such that $Q[x_1,x_2]$ and $Q[x_2,x_1]$ has at most $2$ backward arcs in total.

The Hamilton orcycles
\[
C_1=P'Q[x_1,x_2]y_j, \enspace C_2=P''Q[x_2,x_1]y_i
\]
has at most $11$ backward arcs in total ($7$ from $P'$ and $P''$, $2$ from $Q[x_1,x_2]$ and $Q'[x_2,x_1]$, and one each from $\{x_1y_i,y_ix_1\}$ and $\{x_2y_j,y_jx_2\}$). Hence, at least one of $C_1$ and $C_2$ has at most $5$ backward arcs.
\end{proof}

Now we are ready to prove Theorem \ref{mainHC}. For convenience of the reader, we start from its statement. 

\vspace{2mm}

\noindent{\bf Theorem} {\bf \ref{mainHC}.}
{\em If $D$ is a 2-connected digraph with $\alpha(D) \le 2$, then $D$ contains a Hamilton orcycle with at most 5 backward arcs.}

\vspace{2mm}

\begin{proof}
Let $D$ be a countra-example to the claim of the theorem.
Let $D_X$ be a semicomplete digraph in $D$ with the largest size, and write $X=V(D_X)$. If $|X|=|V(D)|$, then the orcycle obtained by joining the head and the tail of the Hamilton dipath given by Theorem \ref{HamPathTour} give a contradiction. Thus, we can assume that $ Y:= V(D) \setminus X$ is non-empty. 
If $|V(D)| \le 11$, then $D$ is 2-connected and $\alpha(D) \le 2$ implies that $UG(D)$ contains a Hamilton orcycle $C$. So $C$ or its reverse yields a Hamilton orcycle with at most $\lfloor 11/2 \rfloor=5$ backward arcs.
Hence $|V(D)| \ge 12 \ge 9=R(3,4)$, which means that $|X|\ge 4$.

For any set $V \in V(D)$ and any $v \neq v' \in V$, write $Q_V[v,v']$ for any Hamilton $(v,v')$-orpath in $D[V]$ with $b_V(v,v')$ backward arcs. If $V=X$, we simply write $Q[v,v']$.

\textbf{Case 1:} For each vertex $y \in Y$, there is some $x \in X$ such that $xy \in A(D^*)$. 

At this point, the first two conditions of Lemma \ref{DipathPlusSemi} part~(d) are already fulfilled. If $D[Y]$ contains a Hamilton orpath at most $1$ backward arc, then we get a contradiction.

\textbf{Case 1.1:} $D[Y]$ is not connected. 

As $\alpha(D) \le 2$, $D[Y]$ consists of two vertex-disjoint semicomplete digraphs with no arcs between them. We denote them $A$ and $B$. 

First assume that for all $a \in A$ and all $b \in B$, $N_X(a)\cap N_X(b)= \emptyset$. Pick any $a'\in A$, and let $X_1= N_X(a')$.
For all $a \in A $ and all $b \in B$, we have $ab \notin A(D^*)$, so $N_X(a)\cup N_X(b) = X$. Hence, $N_X(b) = X \setminus X_1$ for all $b \in B$. Similarly, $N_X(a)= X_1$ for all $a \in A$. Since $D$ is 2-connected, we must have $|X_1|, |X \setminus X_1| \ge 2$.

Pick any $x_1 \in X_1$ and any $x_2 \in X \setminus X_1$.
Let $X'=X \setminus \{x_1,x_2\}$, $Y'=Y\cup\{x_1,x_2\}$. We have $D[X']$ is semicomplete, $D[Y']$ is connected but not 2-connected (as $x_1$ is a cut vertex of $D[Y']$), and $X' \cup Y'=V(D)$. Note that $|X_1 \setminus \{x_1,x_2\}| \ge 1, |(X \setminus X_1) \setminus \{x_1,x_2\}| \ge 1$ and $x_1, x_2$ are adjacent to all vertices in $X'$, so we have $|N_{X'}(y)| \ge 1$ for all $y \in Y'$. Let $P=y_1y_2\dots y_k$ be a Hamilton orpath in $D[Y']$ with at most $2$ backward arcs, which exists by Theorem \ref{mainHP}. Note that $k=|Y'|=|A|+|B|+2 \ge 4$. As $Y'$ is not 2-connected, we must have $y_1y_k \notin A(D^*)$. Furthermore, $|N_{X'}(y_1)|,|N_{X'}(y_k)| \ge 1$, so we get a contradiction by Lemma \ref{DipathPlusSemi} part~(a).

So there exists $x\in X, a \in A, b \in B$ such that $ax, bx \in A(D^*)$. Let $X'=X \setminus \{x\}$. Note that we can find $a' \in A$ such that $b_A(a,a')+ b_A(a',a)\le 2$. Indeed, if $|A|=1$, this hold trivially with $a'=a$. If $|A| \ge 2$, we apply lemma \ref{lem: FullUnion} with $D \leftarrow D[A], X \leftarrow \{a\}, Y \leftarrow A\setminus\{a\}$ and find $a' \in A \setminus \{a\}$ such that $b_A(a,a')+ b_A(a',a)\le 2$. Similarly, there exist $b' \in B$ such that $b_B(b,b')+ b_B(b',b)\le 2$. 

Since $a'b' \notin A(D^*)$ and $\alpha(D) \le 2$, we have $N_{X}(a') \cup N_{X}(b')= X$, and we may assume that $|N_{X}(b')| \ge N_{X}(a')|$. As $|X| \ge 4$, we have $|N_{X}(b')| \ge 2$ and then $|N_{X'}(b')| \ge 1$. Next we show that we can always find $a_1,a_2 \in A$ such that $xa_1 \in A(D^*)$, and $|N_{X'}(a_2)| \ge 1$, and  if $|A|>1$, $a_1\ne a_2$.
If there exists $a_0 \in A  \setminus \{a\}$ such that $N_{X'}(a_0) \neq \emptyset$, we assign $a_2=a_0$ and assign $a_1=a$. Otherwise, since $|N_X(y)| \ge 1$ for all $y \in Y$, we have $N_{X}(a_0) =\{x\}$ for all $a_0 \in A \setminus a$.
As $D$ is 2-connected, we must have $N_{X'}(a) \neq \emptyset$. In this case we assign $a_2=a$ and either assign $a_1$ as any vertex in $A \setminus \{a_2\}$ if $|A| \ge 2$, or set $a_1=a$ if $|A|=1$. Hence, we have $a_1, a_2$ with the desired properties.

By Lemma \ref{lem: XYpathSemiC}, we have $b_A(a_1,a_2)+ b_A(a_2,a_1)\le 3$.
Since $a_2b' \notin A(D^*)$, we have $N_{X'}(a_2) \cup N_{X'}(b')= X'$. By Lemma \ref{lem: FullUnion}, and since $|N_{X'}(a_2)| \ge 1$, $|N_{X'}(b')| \ge 1$ and $|X'|\ge 3$, there exist $x_a \neq x_b \in X'$ such that $x_a \in N_{X'}(a_2), x_b \in N_{X'}(b')$ and $b_{X'}(x_a,x_b)+ b_{X'}(x_b,x_a)\le 2$.
Consider the following two Hamilton orcycles 
\[
C_1=Q_A(a_2,a_1) x Q_B(b,b') Q_{X'}(x_b,x_a)a_2 \enspace, \quad 
C_2=a_2 Q_{X'}(x_a,x_b)Q_B(b',b) x Q_A(a_1,a_2).  
\]
They have at most 11 backward arcs in total, with at most 3 backward arcs from the union of $Q_A(a_1, a')$ and $Q_A(a', a_1)$,
2 backward arcs from each of $\{Q_{X'}(x_a,x_b), Q_{X'}(x_b,x_a)\}$, $\{Q_B(b,b'), Q_B(b',b)\}$, 
and 1 backward arc from each of $\{a_1x,xa_1\},\{bx,xb\},\{b'x_b,x_b b'\}, \{a_2 x_a, x_a a_2\}$. Hence, either $C_1$ or $C_2$ is a Hamilton orcycle in $D$ with at most 5 backward arcs, a contradiction.

\textbf{Case 1.2:} $D[Y]$ is connected.

\2

\textbf{Case 1.2.1:} There exists $y \in Y$ such that $|Y \setminus (N_Y(y) \cup \{y\})| \le 1$. 

Let $Y'=Y \setminus \{y\}$. By Theorem \ref{PathPartition}, we may partition $D[Y']$ into two dipaths $u_1\dots u_s$ and $v_1 \dots v_t$.
By Lemma \ref{DipathPlusSemi} part~(d), it suffices to find a Hamilton orpath in $D[Y]$ with at most one backward arc to derive a contradiction.

\2

{\bf Claim 0:} {\em $s,t \ge 2$}

\2

{\bf Proof of Claim~0:}
For the sake of contradiction, without loss of generality, assume that $s=1$. If $t=1$, then since $D[Y]$ is connected, $D[Y]$ is an orpath of length 2 with at most one backward arc, a contradiction. Therefore, $t \ge 2$.

If $yv_1 \notin A(D^*)$, then as $|Y \setminus (N_Y(y) \cup \{y\})| \le 1$, we have $u_1y, yv_2, yv_t \in A(D^*)$. Note that $v_1u_1\in  A(D^*)$, since otherwise $v_1 \dots v_t y u_1$ has at most 2 backward arcs, and we get a contradiction by Lemma \ref{DipathPlusSemi} part (a). Also note that $v_1u_1,u_1y, yv_t \in A(D)$, since otherwise $v_ty\in A(D)$ or $yu_1\in A(D)$ or $u_1v_1 \in A(D)$, and the Hamilton orcycle $v_1 \dots v_t y u_1 v_1$ in $D[Y]$ has at most 2 backward arcs, and it contains a Hamilton orpath in $D[Y]$ with at most 1 backward arc, a contradiction. Now we get a Hamilton orpath $v_1 u_1y v_2 \dots v_t$  in $D[Y]$ with at most 1 backward arc, a contradiction.

The case if $yv_t \notin  A(D^*)$ can be proved similarly.

If $u_1y \notin A(D^*)$, then as $|Y \setminus (N_Y(y) \cup \{y\})| \le 1$, we have $yv_i \in A(D^*)$ for all $i=1,2, \dots, t$. By Lemma \ref{lem: PathCombine}, there is a Hamilton dipath $p_1\dots p_{t+1}$ in $D[Y \setminus \{u_1\}]$.
Note that $u_1p_1,u_1p_{t+1} \notin A(D^*)$, since otherwise $u_1p_1 \in A(D^*)$ or $u_1p_{t+1} \in A(D^*)$, and then $u_1p_1p_2\dots p_{t+1}$ or $p_1p_2\dots p_{t+1}u_1$ is a Hamilton orpath in $D[Y]$ with at most one backward arc, a contradiction. As $\alpha(D) \le 2$, we have $p_1p_{t+1} \in A(D^*)$. Since $D[Y]$ is connected, we must have $u_1p_i \in A(D^*)$ for some $i \in 1,2, \dots t+1$. Then either $u_1p_ip_{i+1}\dots p_{t+1}p_1p_2 \dots p_{i-1}$ or $p_{i+1}p_{i+2}\dots p_{t+1}p_1p_2 \dots p_iu_1$ is a Hamilton orpath in $D[Y]$ with at most one backward arc, a contradiction.

Now the case left is $u_1y,yv_1,yv_t \in A(D^*)$. But in this case, either $v_1v_2\dots v_{t}yu_1$ or $u_1y v_1v_2\dots v_{t}$ is a Hamilton orpath in $D[Y]$ with at most one backward arc, depending on the arc between $y$ and $u_1$, which is a contradiction.

This completes the proof of Claim~0.
\2

{\bf Claim 1:} {\em $v_tu_1,u_sv_1 \in A(D^*)$}

\2

{\bf Proof of Claim~1:}
Since $|Y \setminus (N_Y(y) \cup \{y\})| \le 1$, we may assume that $u_sy,v_1y \in A(D^*)$.
$u_1\dots u_syv_1 \dots v_t$ is a Hamilton orpath in $D[Y]$ with at most two backward arcs. If $v_tu_1 \notin A(D^*)$, then we get a contradiction by Lemma \ref{DipathPlusSemi} part~(a). So $v_tu_1 \in A(D^*)$.
If $yv_1 \in A(D)$ or $u_sy \in A(D)$, then $u_1 \dots u_s y v_1 \dots v_t$ is a Hamilton orpath in $D[Y]$ with at most one backward arc, which is also a contradiction. So $v_1y, yu_s \in A(D)$.
Since $|Y \setminus (N_Y(y) \cup \{y\})| \le 1$, $u_{s-1} y \in A(D^*)$ or $yv_2 \in A(D^*)$ . Then $v_1 \dots v_t u_1 \dots u_{s-1} y u_s$ or $v_1 y v_2 \dots v_t u_1 \dots u_s$ is a Hamilton orpath with at most $2$ backward arc. So $v_1u_s \in A(D)$ by Lemma \ref{DipathPlusSemi} part~(a).
This completes the proof of Claim~1.

\2

Since $|Y \setminus (N_Y(y) \cup \{y\})| \le 1$, we may assume that $yv_i \in A(D^*)$ for all $i=1,2, \dots, t$. By Lemma \ref{lem: PathCombine}, there is a Hamilton dipath $p_1,\dots p_{t+1}$ in $D[\{y,v_1,v_2, \dots, v_t\}]$, with $p_1 = v_1$ or $y$ and $p_{t+1} = v_t$ or $y$. 
If $u_1y \in A(D^*)$, then $p_{t+1}u_1 \in A(D^*)$, and so $p_1 \dots p_{t+1}u_1 \dots u_s$ is a Hamilton orpath in $D[Y]$ with at most one backward arc. 
If $u_1y \notin A(D^*)$, then $u_sy \in A(D^*)$, and so $u_sp_1 \in A(D^*)$. Hence, $u_1 \dots u_sp_1 \dots p_{t+1}$ is a Hamilton orpath in $D[Y]$ with at most one backward arc, which is a contradiction. This completes the proof of Case~1.2.1.

\2

\textbf{Case 1.2.2:} $|Y \setminus (N_Y(y) \cup \{y\})| \ge 2$ for all $y \in Y$.

Suppose there exists $y' \in Y$ such that $|N_X(y')|=1$. Then $|X \setminus N_X(y')| = X-1$, and so
\[
|V(D) \setminus (N_D(y')\cup \{y\})|=|X \setminus N_X(y')|+|Y \setminus (N_Y(y')\cup \{y\})|\ge |X|+1.
\]
As $\alpha(D) \le 2$, $D[V(D) \setminus (N_D(y')\cup \{y\})]$ must be a semicomplete digraph larger than $X$, which contradicts the maximality of $X$.
So $|N_X(y)| \ge 2$ for all $y \in Y$. If $Y$ is semicomplete, then we get a contradiction by Corollary \ref{cor: TwoSemiC}. Hence, there exists $y\neq y' \in Y$ such that $yy' \notin A(D^*)$. This implies that $N_X(y) \cup N_X(y') = X $. Since $X$ is a semicomplete digraph of the largest size, we must have $N_X(y) \neq X$, and so $N_X(y) \neq N_X(y')$. 

Let $P=y_1y_2\dots y_k$ be a Hamilton orpath in $D[Y]$ with at most two backward arcs, which exists by Theorem \ref{mainHP}. By Lemma \ref{DipathPlusSemi} part~(a), we have $y_1y_k \in A(D^*)$. Since not all $N_X(y)$ are equal, there must be $i \in \{1,\dots,k-1\}$ such that $N_X(y_i) \neq N_X(y_{i+1})$. 

Pick any $x_{i+1} \in N_X(y_{i+1}) \setminus N_X(y_i)$. Now suppose we have picked $x_j$, we then pick $x_{j+1} \in N_X(y_{j+1}) \setminus \{x_j\}$ (with index modulo $k$). Doing this until we have picked $x_{i-1} \in N_X(y_{i-1}) \setminus \{x_{i-2}\}$. Since $x_{i+1} \notin N_X(y_i)$, we have $|N_X(y_i)\setminus \{x_{i-1},x_{i+1}\}| \ge 1$. So we can pick $x_i \in N_X(y_i)\setminus \{x_{i-1},x_{i+1}\}$.

Now for all $i=1,2,\dots,k$, we have $x_i \notin \{x_{i-1},x_{i+1} \}$. If there is $i \neq j$ such that $x_iy_i, y_jx_j \in A(D)$, then there exists $t$ such that  $y_tx_t, x_{t+1}y_{t+1} \in A(D)$. Then 
\[
Q[x_t,x_{t+1}]y_{t+1}y_{t+2}\dots y_ky_1 \dots y_tx_t
\]
is a Hamilton orcycle in $D$ with at most $5$ backwards arcs, with $2$ possible backward arcs in $Q[x_t,x_{t+1}]$ and $3$ in $y_{t+1}y_{t+2}\dots y_ky_1 \dots y_t$.

If $x_iy_i\in A(D)$ for all $i$, or $y_ix_i\in A(D)$ for all $i$, then 
\[
Q[x_k,x_1]y_1y_2\dots y_kx_k
\]
is a Hamilton orcycle in $D$ with at most 5 backwards arcs, with $2$ possible backward arcs in  $Q[x_k,x_1]$, $2$ in $y_1\dots y_k$, and one of $x_1y_1,y_kx_k$.

This concludes case 1.

\2

\textbf{Case 2:} There is some vertex $y \in Y$ such that $xy \notin A(D^*)$ for all $x \in X$. Let $Z$ contain all such vertices. Since $\alpha(D) \le 2$, 
$D[Z]$ must be a semicomplete digraph.
Let $W=Y \setminus Z$. Consider any vertex $w \in W$. As $\alpha(D)\le 2$ and there is no $(X, Z)$-arc, either $xw \in A(D^*)$ for all $x \in X$ or $zw \in A(D^*)$ for all $z \in Z$. The former case contradicts maximality of $X$ as $D[X \cup \{w\}]$ is a semicomplete digraph bigger than $X$. Hence, for all $w \in W$ and all $z \in Z$, we must have $zw \in A(D^*)$. Furthermore, for all $w \in W$, there is some $x \in X$ such that $xw \in A(D^*)$ (otherwise this $w$ would be in $Z$).

By Theorem \ref{PathPartition}, we can partition $W$ into 2 dipaths $u_1 \dots u_s$ and $v_1 \dots v_t$. By Theorem \ref{HamPathTour}, there exists a Hamilton dipath $P_z=p_1 \dots p_l$ in $Z$.

\begin{figure}
\begin{tikzpicture}[
    bigset/.style={draw, circle, minimum size=28mm},
    smallv/.style={draw, circle, minimum size=3mm, inner sep=0pt},
    Wbox/.style={draw, rounded corners=6pt, fit=#1, inner sep=10mm},
    >=Stealth
]

\node[bigset] (X) at (0,-3) {};
\node at (0,-5) {$X$};

\node[bigset] (Z) at (9.5,-3) {};
\node at (9.5,-5) {$Z$};

\node[smallv] (u1) at (3.2,1.4) {};
\node[smallv] (u2) at (4.0,1.4) {};
\node[smallv] (us) at (6.0,1.4) {};

\node at (3.2,0.8) {$u_1$};
\node at (4.0,0.8) {$u_2$};
\node at (6.0,0.8) {$u_s$};

\draw[->] (u1) -- (u2);
\draw[->] (u2) -- ($(u2)+(0.7,0)$);
\draw[->] ($(us)-(0.7,0)$) -- (us);

\node at (5.0,1.4) {$\cdots$};

\node[smallv] (v1) at (3.2,0.2) {};
\node[smallv] (v2) at (4.0,0.2) {};
\node[smallv] (vt) at (6.0,0.2) {};

\node at (3.2,-0.4) {$v_1$};
\node at (4.0,-0.4) {$v_2$};
\node at (6.0,-0.4) {$v_t$};

\draw[->] (v1) -- (v2);
\draw[->] (v2) -- ($(v2)+(0.7,0)$);
\draw[->] ($(vt)-(0.7,0)$) -- (vt);

\node at (5.0,0.2) {$\cdots$};

\node[Wbox={(u1)(us)(v1)(vt)}] (W) {};
\node at ($(W.south)+(0,-0.6)$) {$W$};

\def\rX{14mm}

\coordinate (Ztop) at ($(Z.center)+(\rX*cos 90,\rX*sin 90)$);
\coordinate (Zbot) at ($(Z.center)+(\rX*cos 180,\rX*sin 180)$);

\coordinate (Wtop) at (W.east);
\coordinate (Wbot) at ($(W.south)+(1,0)$);

\draw (Ztop) -- (Wtop);
\draw (Zbot) -- (Wbot);
\draw (Ztop) -- (Wbot);
\draw (Zbot) -- (Wtop);

\draw (X.north) -- (W.west);
\draw (X.east) -- ($(W.south)-(1,0)$);

\end{tikzpicture}
    \centering
    \caption{Case 2}
    \label{fig:placeholder}
\end{figure}

\2

{\bf Claim A:} {$v_tu_1,u_sv_1 \in A(D^*)$}

\2

{\bf Proof of Claim~A:}
$u_1 \dots u_sP_zv_1 \dots v_t$ is a Hamilton orpath in $D[Y]$ with at most $2$ backward arcs. If $v_tu_1 \notin A(D^*)$, then we get a contradiction by Lemma \ref{DipathPlusSemi} part~(a). Hence, $v_tu_1 \in A(D^*)$.
Similarly, we also have $u_sv_1 \in A(D^*)$.

\2

{\bf Claim B:} {The following conditions hold.
\begin{description}
\item[(1)] For all $x' \in N_X(u_s)$ and all $x \in N_X(v_1)\setminus \{x'\}$, we have $v_1x \in A(D)$.

\item[(2)] For all $x' \in N_X(v_1)$ and all $x \in N_X(u_s)\setminus \{x'\}$, we have $xu_s \in A(D)$.

\item[(3)] For all $x' \in N_X(v_t)$ and all $x \in N_X(u_1)\setminus \{x'\}$, we have $u_1x \in A(D)$.

\item[(4)] For all $x' \in N_X(u_1)$ and all $x \in N_X(v_t)\setminus \{x'\}$, we have $xv_t \in A(D)$.
\end{description}
}

\2
{\bf Proof of Claim~B:}
We first prove part~(1). 
Consider any $x, x'$ that satisfy the hypothesis. If $xv_1  \in A(D)$, then
\[
u_1 \dots u_sQ[x',x]v_1 \dots v_t p_1\dots p_lu_1,
\]
is a Hamilton orcycle in $D$ with at most $5$ backward arcs.
The possible backward arcs are $v_tp_1,p_lu_1,u_sx'$, and $2$ arcs from $Q[x',x]$.

Similarly, suppose $x, x'$ satisfy the hypothesis of part (2).
If $u_sx  \in A(D)$, then
\[
u_1 \dots u_sQ[x,x']v_1 \dots v_t p_1\dots p_lu_1,
\]
is a Hamilton orcycle in $D$ with at most $5$ backward arcs.
The possible backward arcs are $v_tp_1,p_lu_1,x'v_1$, and $2$ arcs from $Q[x',x]$.

Parts (3) and (4) can be proved analogously. This proved Claim~B.

\2

{\bf Claim C:} {$|X| \ge 5$ or $|Z|=1$.}

\2

{\bf Proof of Claim~C:}
Suppose that $|X|\le 4$. As any semicomplete digraph $W'$ in $W$ together with all vertices in $Z$ create a semicomplete digraph of size $|W'|+|Z|$, which must be at most $|X|\le 4$. So we must have $|W| \le R(3,5-|Z|)-1$. If $|Z| \ge 2$, then $|Z|+|W| \le 7$, so $|V(D)| \le |X|+|Z|+|W| \le 11$, which contradicts  $|V(D)| \ge 12$.
So we may assume that $|Z|=1$ as $|Z|=0$ is the \textbf{Case 1}.
This proves Claim~C.

\2
By Claim~A, we have a Hamilton orcycle $C_W=u_1 \dots u_sv_1 \dots v_tu_1$ in $D[W]$ with at most $2$ backward arcs. We may refer to $C_W$ as $C_W=w_1 \dots w_{s+t}w_1$, with $w_{s+t+1}$ means $w_1$. Note that $w_i=u_i$ for $i=1,2, \dots, s$.

{\bf Claim D:} {$|N_X(w)|\le 2$ for all $w \in W$. Furthermore, if $|N_X(w_j)|= 2$, then we may assume that  $|N_X(w_{j-1})| = |N_X(w_{j+1})|= 1$.}

\2

{\bf Proof of Claim~D:}
Suppose the claim is false. Without loss of generality, we may assume that there exists a vertex $u_i$ for some $i \in \{1,2,\dots,s\}$ which contradicts the assumption. So either $|N_X(w_i)|\ge 3$, or $|N_X(w_i)|= 2$ and $ |N_X(w_{i-1})| \ge 2$, or $|N_X(w_i)|= 2$ and $|N_X(w_{i+1})| \ge 2$.

If $|N_X(w_i)|\ge 3$, we can pick any $x_1 \in N_X(w_{i-1}), x_3 \in N_X(w_{i+1})$, and has $|N_X(w_i)\setminus \{x_1,x_3\}|\ge 1$. So we can pick $x_2 \in N_X(w)\setminus \{x_1,x_3\}$.
If $|N_X(w_i)|= 2$ and $|N_X(w_{i-1})| \ge 2$, pick $x_3 \in N_X(w_{i+1}), x_2 \in N_X(w_i)\setminus \{x_3\}, x_1 \in N_X(w_{i-1}) \setminus \{x_2\}$.
If $|N_X(w_i)|= 2$ and $|N_X(w_{i+1})| \ge 2$, pick $x_1 \in N_X(w_{i-1}), x_2 \in N_X(w_i)\setminus \{x_1\}, x_3 \in N_X(w_{i+1}) \setminus \{x_2\}$.

So, in all cases, we can pick $x_1 \in N_X(w_{i-1}), x_2 \in N_X(w_i), x_3 \in N_X(w_{i+1})$ such that $x_2 \notin \{ x_1,x_3 \}$. By Lemma \ref{lem: PathCombine}, we can combine $v_1 \dots v_t$ and $p_1 \dots p_l$ into one dipath $q_1q_2 \dots q_{t+l}$ , with $q_1= v_1$ or $p_1$, and $q_{t+l}=v_t$ or $p_l$. 
Since $p_1,p_l \in Z$, $u_1,u_s \in W$, and by Claim~A, we have $v_1u_s,p_1u_s,v_tu_1,p_lu_1 \in A(D^*)$.
So $q_1u_s,q_{t+l}u_1 \in A(D^*)$.
We have a Hamilton orcycle $C'=q_1q_2 \dots q_{t+l}u_1\dots u_s q_1$ in $D[Y]$ with at most $2$ backward arcs ($q_{t+l}u_1$ and $u_sq_1$).

\2

\textbf{Case D1:} $i \in \{2,\dots,s-1\}$. 
We have $w_{i-1}=u_{i-1}, w_i=u_i, w_{i+1}=u_{i+1}$. As $q_1q_2 \dots q_{t+l}u_1\dots u_s$ is a Hamilton orpath in $D[Y]$ with at most one backward arc, we get a contradiction by Lemma \ref{DipathPlusSemi} part~(c).

\2

\textbf{Case D2:} $i=s, s \ge 2$. We have $w_{i-1}=u_{s-1}, w_i=u_s, w_{i+1}=v_1$. By Claim~B part~(2), we have $x_2u_s \in A(D)$.
The orcycle
\[
C'[u_s,u_1]u_2 \dots u_{s-1} Q[x_1,x_2]u_s,
\]
is a Hamilton orcycle in $D$ with at most $5$ backward arcs. 
The possible backward arcs are $2$ in $C'[u_s,u_1]$, $2$ in $Q[x_1,x_2]$, and $u_{s-1}x_1$. Note that if $s=2$, the orcycle simply degenerated to $C'[u_2,u_1]Q[x_1,x_2]u_2$.

\2

\textbf{Case D3:} $i=1, s\ge 2$. We have $w_{i-1}=v_t, w_i=u_1, w_{i+1}=u_2$. By Claim~B part~(3), we have $u_1x_2 \in A(D)$. The orcycle 
\[
C'[u_s,u_1]Q[x_2,x_3]u_2 \dots u_s,
\]
is a Hamilton orcycle in $D$ with at most $5$ backward arcs.
The possible backward arcs are $2$ in $C'[u_s,u_1]$, $2$ in $Q[x_2,x_3]$, and $x_3u_2$. Note that if $s=2$, the orcycle simply degenerated to $C'[u_2,u_1]Q[x_2,x_3]$.

\2

\textbf{Case D4:} $i=s=1$. In this case, we have $w_{i-1}=v_t, w_i=u_1, w_{j+1}=v_1$. By Claim~B part~(2), we have $x_2u_1 \in A(D)$, but by Claim~B part~(3), we have $u_1x_2 \in A(D)$, a contradiction.
This proved Claim~D.

\2

Since $Y=Z \cup W$, $D[Z]$ is semicomplete, and $zw \in A(D^*)$ for all $w \in W, z \in Z$, if $D[W]$ is semicomplete, then $D[Y]$ is also semicomplete.
By Lemma \ref{lem: TwoSemiC}, we may assume that $D[W]$ is not semicomplete. So there exists $w \neq w' \in W$ such that $ww' \notin A(D^*)$. 

Since $\alpha(D) \le 2$, we have $N_X(w)\cup N_X(w')=X$. If $|X| \ge 5$, then either $|N_X(w)|\ge 3$ or $|N_X(w')|\ge 3$. Both contradict Claim~D. So by Claim~C, we must have $|Z|=1$ and $|X|=4$.

Since $|V(D)| \ge 12$, we have $|W| \ge 12-|X|-|Z| =7$. By Claim~D, every vertex in $W$ has degree at most $2$ in $D[X]$, and at least half of them must have degree $1$. Hence, at least $4$ vertices in $W$ have degree exactly $1$ in $X$. If two distinct vertices $w_1,w_2 \in W$ have degree $1$ in $X$, then there is some $x \in X \setminus (N_X(w_1) \cup N_X(w_2))$. Since $\alpha(D) \le 2$, we have $w_1w_2 \in A(D^*)$. 
Then we know that the vertex in $Z$ together with $4$ vertices with degree $1$ in $W$  induce a semicomplete digraph of size $5 > |X|$, which contradicts that $X$ is one of the largest semicomplete digraphs in $D$.
\end{proof}

\section{Open problems on Hamilton orpaths and orcycles in digraphs with independence number at least 2}\label{sec:OP}

Theorems \ref{mainHP} and \ref{mainHC} show that connected (2-connected, resp.) digraphs $D$ with $\alpha(D)=2$ have a Hamilton orpath (a Hamilton orcycle, resp.) with a constant number of backward arcs. Perhaps, these results can be extended to the case of $\alpha(D)=k\ge 3$. 

\begin{problem}\label{op1}
Let $D'$ ($D''$, resp.) be a digraph which has a Hamilton orpath (orcycle, resp.) and let
$b_P(D')$ ($b_C(D'')$, resp.) be the minimum number of backward arcs in a Hamilton orpath (orcycle, resp.) in $D'$ ($D''$, resp.). Are there functions $\beta_P(k)$ and $\beta_C(k)$ such that if $\alpha(D')=\alpha(D'')=k$ then $b_P(D')\le \beta_P(k)$ and $b_C(D'')\le \beta_C(k)$?
\end{problem} 

The following simple reduction shows that if there is a function $\beta_C(k)$, then there is a function $\beta_P(k)$ and 
 $\beta_P(k)<\beta_C(k)$. 
Indeed, let $\beta_C(k)$ exist for each $k$ and let  $D'$ be a digraph which has a Hamilton orpath and $\alpha(D')=k$. Add a new vertex $x$ to $D'$ together with all possible arcs from $x$ to $V(D')$; denote the resulting digraph $D^*.$ Observe that 
$\alpha(D^*)=k$, $D^*$ has a Hamilton orcycle $Z$ with at most $\beta_C(k)$ backward arcs. Hence, the Hamilton orpath $Z-x$ of $D'$ has at most $\beta_C(k)-1$ backward arcs. 

Hence, if the answer to Problem \ref{op1} is positive, then it is enough to consider the Hamilton orcycle subproblem. 
The next result establishes a lower bound on $\beta_C(k)$, if it exists.

\begin{theorem}\label{lowbound4k}
For every integer $k\ge 2$, there is a digraph $D_k$ with $\alpha(D_k)=k$ such that the number of backward arcs in every Hamilton orcycle of $D_k$ is at least $\lfloor 5k/2 \rfloor -1$.
\end{theorem}
\begin{proof}
Let $Q=TT(v_1,v_2,v_3,v_4,v_5)$ be a (transitive) tournament on five vertices, $v_1,v_2,v_3,v_4,v_5$, such that $v_iv_j \in A(Q)$ if and only if $i<j$. Observe the minimum number of backward arcs in a Hamilton $(v_2,v_4)$-orpath is exactly one. Indeed, 
$v_2v_3v_5v_1v_4$ is a Hamilton $(v_2,v_4)$-orpath with one backward arc and $v_5$ must be incident with a backward arc in any Hamilton $(v_2,v_4)$-orpath. Also,  observe that the minimum number of backward arcs in a Hamilton $(v_4,v_2)$-orpath $Z$ is two. Indeed, for every Hamilton $(v_4,v_2)$-orpath in $Q$ there must be a backward arc on the subpath from $v_4$ to $v_3$ and also a backward arc on the subpath from $v_3$ to $v_2$.
 Finally, observe that $v_4v_5v_1v_3v_2$ is a Hamilton $(v_4,v_2)$-orpath with two backward arcs.

We now construct a digraph $D_k$ with $\alpha(D_k)=k$ as follows ($k \geq 2$); see Figure \ref{figD6} for $D_6.$
Start with $k$ copies of $Q$, named $Q_1,Q_2,\ldots ,Q_k$. Let $x_i,y_i \in V(Q_i)$ be chosen as follows.
\begin{itemize}
\item If $i \leq \frac{k}{2}$, let $Q_i=TT(u,x_i,v,y_i,w)$; 
\item If $i > \frac{k}{2}$,   let $Q_i=TT(w,y_i,v,x_i,u)$. 
\end{itemize}
Then add any tournament $T$ on $k-2$ vertices $u_1,u_2,\ldots,u_{k-2}$ and add any arcs between $T$ and $Q_k$ such that $V(Q_k) \cup V(T)$ induces a tournament in $D_k$. Finally add the following arcs.
\begin{description}
\item[(a)] $y_i u_i$, $x_{i+1} u_i$ for all $i = 1,2, \ldots , k-2$;
\item[(b)] $x_1 y_k$ and $y_{k-1} x_k$.
\end{description}
This completes the construction of $D_k$. Note that $D_k$ can be decomposed into $k$ tournaments, $Q_1, Q_2, \ldots, Q_{k-1}, D_k[V(Q_k) \cup V(T)]$,
which implies that $\alpha(D_k)=k$.

We will now show that any Hamilton orcycle in $D_k$ has at least $\lfloor 5k/2 \rfloor -1$ backward arcs.
Let $C$ be a Hamilton orcycle in $D_k$ which has the minimum number of backward arcs. As $C$ has to enter and leave $Q_i$ for all $i \in \{1,2,\ldots,k-1\}$ we note
that all arcs defined in (a) and (b) above belong to $C$ (some as forward arcs and some as backward arcs). This means that exactly one of the following two cases holds:

\begin{description}
\item[Case 1:] $C$ contains a Hamilton $(x_i,y_i)$-orpath in $Q_i$ for all $i \in \{1,2,\ldots,k-1\}$;

\item[Case 2:]  $C$ contains a Hamilton $(y_i,x_i)$-orpath in $Q_i$ for all $i \in \{1,2,\ldots,k-1\}$.

\end{description}

For $i=1,2,\dots ,k-2$,  each $u_i$ is incident with exactly two arcs in $C$, which must be  $y_i u_i$ and $x_{i+1} u_i$. Hence, no arcs between $T$ and $Q_k$ belong to $C$. Therefore, if Case 1 holds, then $C$ contains
a Hamilton $(x_k,y_k)$-orpath in $Q_k$; otherwise, Case 2 holds and $C$ contains a Hamilton $(y_k,x_k)$-orpath in $Q_k$. Whatever Case 1 or Case 2 holds, these Hamilton orpaths in all $Q_i$'s which are contained in $C$, has at least a total of $2\lfloor k/2 \rfloor + \lceil k/2 \rceil=\lfloor 3k/2 \rfloor$ backward arcs.
However, half the arcs defined in (a) and (b) above will also be backward arcs in $C$, so the total number of backward arcs in $C$ is at least $\lfloor 3k/2 \rfloor + k-1 = \lfloor 5k/2 \rfloor -1$.
\end{proof}

Note that when $k=2$ the above gives an example where every Hamilton orcycle has at least four backward arcs.

\begin{figure}[htbp]
 \centering  
\begin{tikzpicture}[node distance=1.5cm,scale=0.9, 
    smallset/.style={ellipse, draw, minimum height=1.2cm, minimum width=3.5cm},
    medset/.style={ellipse, draw, minimum height=1.2cm, minimum width=2.5cm}]

    \draw (0,0) node[fill=black,minimum size =2pt,circle,inner sep=1pt] (v11) {};
    \draw (0.5,0) node[fill=black,minimum size =2pt,circle,inner sep=1pt,label=above:$x_1$] (v12) {};
    \draw (1,0) node[fill=black,minimum size =2pt,circle,inner sep=1pt] (v13) {};    
    \draw (1.5,0) node[fill=black,minimum size =2pt,circle,inner sep=1pt,label=above:$y_1$] (v14) {};
    \draw (2,0) node[fill=black,minimum size =2pt,circle,inner sep=1pt] (v15) {};

    \draw[arrows={-Stealth[reversed, reversed]}] (v11) -- (v12);
    \draw[arrows={-Stealth[reversed, reversed]}] (v12) -- (v13);
    \draw[arrows={-Stealth[reversed, reversed]}] (v13) -- (v14);
    \draw[arrows={-Stealth[reversed, reversed]}] (v14) -- (v15);

    \node[medset] (Q1) at (1,0) {};
    \node at (1,1) {\large $Q_1$};

    \draw (4,0) node[fill=black,minimum size =2pt,circle,inner sep=1pt] (v21) {};
    \draw (4.5,0) node[fill=black,minimum size =2pt,circle,inner sep=1pt,label=above:$x_2$] (v22) {};
    \draw (5,0) node[fill=black,minimum size =2pt,circle,inner sep=1pt] (v23) {};    
    \draw (5.5,0) node[fill=black,minimum size =2pt,circle,inner sep=1pt,label=above:$y_2$] (v24) {};
    \draw (6,0) node[fill=black,minimum size =2pt,circle,inner sep=1pt] (v25) {};

    \draw[arrows={-Stealth[reversed, reversed]}] (v21) -- (v22);
    \draw[arrows={-Stealth[reversed, reversed]}] (v22) -- (v23);
    \draw[arrows={-Stealth[reversed, reversed]}] (v23) -- (v24);
    \draw[arrows={-Stealth[reversed, reversed]}] (v24) -- (v25);

    \node[medset] (Q2) at (5,0) {};
    \node at (5,1) {\large $Q_2$};

    \draw (8,0) node[fill=black,minimum size =2pt,circle,inner sep=1pt] (v31) {};
    \draw (8.5,0) node[fill=black,minimum size =2pt,circle,inner sep=1pt,label=above:$x_3$] (v32) {};
    \draw (9,0) node[fill=black,minimum size =2pt,circle,inner sep=1pt] (v33) {};    
    \draw (9.5,0) node[fill=black,minimum size =2pt,circle,inner sep=1pt,label=above:$y_3$] (v34) {};
    \draw (10,0) node[fill=black,minimum size =2pt,circle,inner sep=1pt] (v35) {};

    \draw[arrows={-Stealth[reversed, reversed]}] (v31) -- (v32);
    \draw[arrows={-Stealth[reversed, reversed]}] (v32) -- (v33);
    \draw[arrows={-Stealth[reversed, reversed]}] (v33) -- (v34);
    \draw[arrows={-Stealth[reversed, reversed]}] (v34) -- (v35);

    \node[medset] (Q3) at (9,0) {};
    \node at (9,1) {\large $Q_3$};

    \draw (3.5,-2) node[fill=black,minimum size =2pt,circle,inner sep=1pt,label=below:$u_1$] (u1) {};
    \draw (4.5,-2) node[fill=black,minimum size =2pt,circle,inner sep=1pt,label=below:$u_2$] (u2) {};
    \draw (5.5,-2) node[fill=black,minimum size =2pt,circle,inner sep=1pt,label=below:$u_3$] (u3) {};    
    \draw (6.5,-2) node[fill=black,minimum size =2pt,circle,inner sep=1pt,label=below:$u_4$] (u4) {};

    \draw (u1) -- (u2);
    \draw (u2) -- (u3);
    \draw (u3) -- (u4);

    \node[smallset] (T) at (5,-2) {};
    \node at (2.5,-2) {\large $T$};

    \draw (0,-4) node[fill=black,minimum size =2pt,circle,inner sep=1pt] (v61) {};
    \draw (0.5,-4) node[fill=black,minimum size =2pt,circle,inner sep=1pt,label=below:$y_6$] (v62) {};
    \draw (1,-4) node[fill=black,minimum size =2pt,circle,inner sep=1pt] (v63) {};    
    \draw (1.5,-4) node[fill=black,minimum size =2pt,circle,inner sep=1pt,label=below:$x_6$] (v64) {};
    \draw (2,-4) node[fill=black,minimum size =2pt,circle,inner sep=1pt] (v65) {};

    \draw[arrows={-Stealth[reversed, reversed]}] (v61) -- (v62);
    \draw[arrows={-Stealth[reversed, reversed]}] (v62) -- (v63);
    \draw[arrows={-Stealth[reversed, reversed]}] (v63) -- (v64);
    \draw[arrows={-Stealth[reversed, reversed]}] (v64) -- (v65);

    \node[medset] (Q6) at (1,-4) {};
    \node at (1,-5) {\large $Q_6$};

    \draw (4,-4) node[fill=black,minimum size =2pt,circle,inner sep=1pt] (v51) {};
    \draw (4.5,-4) node[fill=black,minimum size =2pt,circle,inner sep=1pt,label=below:$y_5$] (v52) {};
    \draw (5,-4) node[fill=black,minimum size =2pt,circle,inner sep=1pt] (v53) {};    
    \draw (5.5,-4) node[fill=black,minimum size =2pt,circle,inner sep=1pt,label=below:$x_5$] (v54) {};
    \draw (6,-4) node[fill=black,minimum size =2pt,circle,inner sep=1pt] (v55) {};

    \draw[arrows={-Stealth[reversed, reversed]}] (v51) -- (v52);
    \draw[arrows={-Stealth[reversed, reversed]}] (v52) -- (v53);
    \draw[arrows={-Stealth[reversed, reversed]}] (v53) -- (v54);
    \draw[arrows={-Stealth[reversed, reversed]}] (v54) -- (v55);

    \node[medset] (Q5) at (5,-4) {};
    \node at (5,-5) {\large $Q_5$};

    \draw (8,-4) node[fill=black,minimum size =2pt,circle,inner sep=1pt] (v41) {};
    \draw (8.5,-4) node[fill=black,minimum size =2pt,circle,inner sep=1pt,label=below:$y_4$] (v42) {};
    \draw (9,-4) node[fill=black,minimum size =2pt,circle,inner sep=1pt] (v43) {};    
    \draw (9.5,-4) node[fill=black,minimum size =2pt,circle,inner sep=1pt,label=below:$x_4$] (v44) {};
    \draw (10,-4) node[fill=black,minimum size =2pt,circle,inner sep=1pt] (v45) {};

    \draw[arrows={-Stealth[reversed, reversed]}] (v41) -- (v42);
    \draw[arrows={-Stealth[reversed, reversed]}] (v42) -- (v43);
    \draw[arrows={-Stealth[reversed, reversed]}] (v43) -- (v44);
    \draw[arrows={-Stealth[reversed, reversed]}] (v44) -- (v45);

    \node[medset] (Q3) at (9,-4) {};
    \node at (9,-5) {\large $Q_4$};

    \draw[arrows={-Stealth[reversed, reversed]}] (v14) -- (u1);
    \draw[arrows={-Stealth[reversed, reversed]}] (v22) -- (u1);
    \draw[arrows={-Stealth[reversed, reversed]}] (v24) -- (u2);
    \draw[arrows={-Stealth[reversed, reversed]}] (v32) -- (u2);
    \draw[arrows={-Stealth[reversed, reversed]}] (v34) -- (u3);
    \draw[arrows={-Stealth[reversed, reversed]}] (v44) -- (u3);
    \draw[arrows={-Stealth[reversed, reversed]}] (v42) -- (u4);
    \draw[arrows={-Stealth[reversed, reversed]}] (v54) -- (u4);
    \draw[arrows={-Stealth[reversed, reversed]}] (v12) -- (v62);
    \draw[arrows={-Stealth[reversed, reversed]}] (v52) to[out=140, in=40]  (v64);

    \draw[dashed] (0,-3.5) -- (3.5,-2.5);
    \draw[dashed] (1,-3.5) -- (4.5,-2.5);
    \draw[dashed] (1.5,-3.5) -- (5.5,-2.5);
    \draw[dashed] (2,-3.5) -- (6,-2.5);

\end{tikzpicture}
\caption{An example when $k=6$. The dashed lines indicate  $T$ and $Q_6$ are fully connected.}
\label{figD6}
\end{figure}
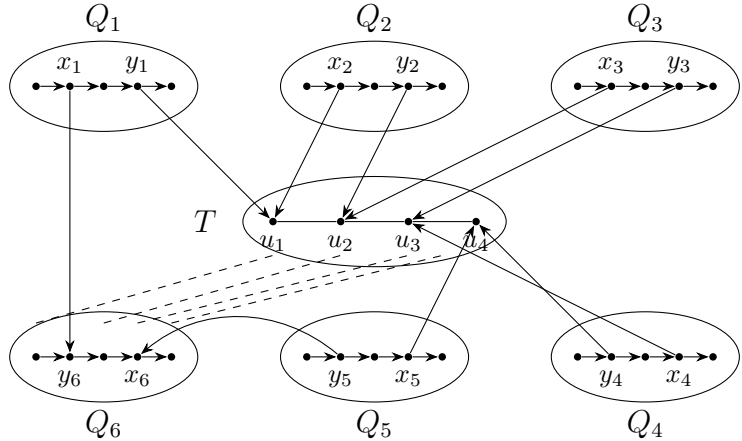

Recall that the following problem remains open.

\begin{problem}
Is it true that every 2-connected digraph $D$ with $\alpha(D)\le 2$ has a Hamilton orcycle with at most four backward arcs?
\end{problem}

By theorems of R{\'e}dei and Camion stated in Section \ref{sec:intro}, in polynomial time, we can decide the minimum number of backward arcs in a Hamilton orcycle (orpath, respectively) of a semicomplete digraph. 
We do not know the computational complexity of these problems for digraphs $D$ with $\alpha(D)\le 2$.

\begin{problem}
Let $D$ be a connected digraph with $\alpha(D)\le 2$. What is the complexity of deciding what is the minimum number of backward arcs in a Hamilton orpath of $D$? 
\end{problem}

\begin{problem}
Let $D$ be a 2-connected digraph with $\alpha(D)\le 2$. What is the complexity of deciding what is the minimum number of backward arcs in a Hamilton orcycle of $D$? 
\end{problem}

\end{document}